\documentclass[a4paper,fleqn,11pt]{amsart}

\usepackage{mathrsfs}
\usepackage{mathptmx}
\usepackage{graphicx}
\usepackage[T1]{fontenc}
\usepackage{cite}
\usepackage{amssymb}
\usepackage[all]{xy}
\usepackage{fancybox}
\usepackage{float}
\usepackage{amsmath}

\DeclareMathAlphabet{\mathpzc}{OT1}{pzc}{m}{it}

\newtheorem{theorem}{Theorem}[section]

\newtheorem{corollary}[theorem]{Corollary}

\theoremstyle{definition}
\newtheorem{definition}[theorem]{Definition}

\newtheorem{notation}[theorem]{Notation}

\theoremstyle{remark}
\newtheorem{remark}[theorem]{Remark}

\numberwithin{equation}{section}

\allowdisplaybreaks[4]

\begin{document}

\title{Uniform stable radius, L\^e numbers and topological triviality for line singularities}

\author{Christophe Eyral}

\address{C. Eyral, Institute of Mathematics, Polish Academy of Sciences, 
\'Sniadeckich~8, 00-656 Warsaw, Poland}  
\email{eyralchr@yahoo.com} 


\begin{abstract}
Let $\{f_t\}$ be a family of complex polynomial functions with \emph{line} singularities. We show that if $\{f_t\}$ has a \emph{uniform stable radius} (for the corresponding Milnor fibrations), then the L\^e numbers of the functions $f_t$ are independent of $t$ for all small $t$. In the case of isolated singularities --- a case for which the only non-zero L\^e number coincides with the Milnor number --- a similar assertion was proved by M. Oka and D. O'Shea.

By combining our result with a theorem of J.~Fern\'andez de Bobadilla --- which says that families of line singularities in $\mathbb{C}^n$, $n\geq 5$, with constant L\^e numbers are topologically trivial --- it follows that a family of line singularities in $\mathbb{C}^n$, $n\geq 5$, is topologically trivial if it has a uniform stable radius. 

As an important example, we show that families of weighted homogeneous line singularities have a uniform stable radius if the nearby fibres $f_t^{-1}(\eta)$, $\eta\not=0$, are ``uniformly'' non-singular with respect to the deformation parameter $t$.
\end{abstract}

\maketitle

\markboth{C. Eyral}{Uniform stable radius, L\^e numbers and topological triviality for line singularities}  

\section{Introduction}

Let $(t,\mathbf{z}):=(t,z_1,\ldots,z_n)$ be linear coordinates for $\mathbb{C}\times \mathbb{C}^n$ ($n\geq 2$), and let 
\begin{equation}\label{lfi}
f\colon (\mathbb{C} \times \mathbb{C}^n,\mathbb{C} \times \{\mathbf{0}\}) \rightarrow (\mathbb{C},0),\
(t,\mathbf{z})\mapsto f(t,\mathbf{z}), 
\end{equation}
be a polynomial function.
As usual, we write $f_t(\mathbf{z}):= f(t,\mathbf{z})$, and for any $\eta\in\mathbb{C}$, we denote by $V(f_t-\eta)$ the hypersurface in $\mathbb{C}^n$ defined by the equation $f_t(\mathbf{z})=\eta$. (Note that (\ref{lfi}) implies $f_t(\mathbf{0}) = f(t,\mathbf{0})=0$, so that the origin $\mathbf{0}\in \mathbb{C}^n$ belongs to the hypersurface $V(f_t)=f_t^{-1}(0)$ for all $t\in\mathbb{C}$.) 

The purpose of this paper is to show that if the polynomial function $f$ defines a family $\{f_t\}$ of hypersurfaces with \emph{line} singularities and with a \emph{uniform stable radius} (for the corresponding Milnor fibrations),
then the L\^e numbers 
\begin{equation*}
\lambda^{0}_{f_t,\mathbf{z}}(\mathbf{0}),\ldots, \lambda^{n-1}_{f_t,\mathbf{z}}(\mathbf{0})
\end{equation*}
of the polynomial functions $f_t$ at $\mathbf{0}$ with respect to the coordinates $\mathbf{z}$ --- which do exist in this case --- are independent of $t$ for all small $t$ (cf.~Theorem \ref{mt1}). In the case of hypersurfaces with \emph{isolated} singularities --- a case for which the constancy of the L\^e numbers means the constancy of the Milnor number --- a similar assertion was proved by M.~Oka \cite{O2} and D. O'Shea \cite{OS}. 

It is worth to observe that by combining Theorem \ref{mt1} with a theorem of J. Fern\'andez de Bobadilla \cite{F} --- which says that a family of hypersurfaces with line singularities in $\mathbb{C}^n$, $n\geq 5$, is topologically trivial if it has constant L\^e numbers --- it follows that a family of hypersurfaces with line singularities in $\mathbb{C}^n$, $n\geq 5$, is topologically trivial if it has a uniform stable radius (cf.~Corollary \ref{mt2}).

It is well known that if $\{f_t\}$ is a family of \emph{isolated} hypersurface singularities such that each $f_t$ is \emph{weighted homogeneous} with respect to a given system of weights, then $\{f_t\}$ has a uniform stable radius --- a result of M. Oka \cite{O2} and D.~O'Shea \cite{OS}.
In Theorem \ref{mt3}, we show this still holds true for weighted homogeneous hypersurfaces with \emph{line} singularities provided that the nearby fibres $V(f_t-\eta)$, $\eta\not=0$, are ``uniformly'' non-singular with respect to the deformation parameter $t$ --- that is, non-singular in a small ball the radius of which does not depends on $t$. (Note that this condition always holds true for isolated singularities.) In particular, by Corollary \ref{mt2}, such families have constant L\^e numbers, and for $n\geq 5$, they are topologically trivial.

Finally, let us observe that by combining Corollary \ref{mt2} with a theorem of M. Oka \cite{O} --- which says that a family $\{f_t\}$ of non-degenerate functions with constant Newton boundary has a uniform stable radius --- we get a new proof of a theorem of J.~Damon \cite{D} which says that if $\{f_t\}$ is a family of non-degenerate line singularities in $\mathbb{C}^n$, $n\geq 5$, with constant Newton boundary, then $\{f_t\}$ is topologically trivial.

\begin{notation}
In this paper, we are only interested in the behaviour of functions (or hypersurfaces) near the origin $\mathbf{0}\in\mathbb{C}^n$.
We denote by $B_\varepsilon$ the closed ball centred at $\mathbf{0}\in\mathbb{C}^n$ with radius $\varepsilon>0$, and we write $\mathring{B}_\varepsilon$ (respectively, $S_\varepsilon$) for its interior (respectively, its boundary). As usual, in $\mathbb{C}$, we rather write $D_\varepsilon$ and $\mathring{D}_\varepsilon$  instead of $B_\varepsilon$ and~$\mathring{B}_\varepsilon$. 
\end{notation}

\section{Uniform stable radius}\label{Sect-USR}

By \cite[Lemme (2.1.4)]{HL}, we know that for each $t$ there exists a positive number $r_t>0$ such that for any pair $(\varepsilon_t,\varepsilon'_t)$ with $0<\varepsilon'_t\leq \varepsilon_t\leq r_t$, there exists $\delta(\varepsilon_t,\varepsilon'_t)>0$ such that for any non-zero complex number $\eta$ with $0<\vert\eta\vert\leq\delta(\varepsilon_t,\varepsilon'_t)$, the hypersurface $V(f_t-\eta)$ is non-singular in $\mathring{B}_{r_t}$ and transversely intersects with the sphere $S_{\varepsilon''}$ for any $ \varepsilon''$ with $\varepsilon'_t\leq \varepsilon''\leq \varepsilon_t$. Any such a number $r_t$ is called a \emph{stable radius} for the Milnor fibration of $f_t$ at $\mathbf{0}$ (cf.~\cite[\S 2]{O}).

\begin{definition}[\mbox{cf.~\cite[\S 3]{O}}]\label{def-WUSR}
We say that the family $\{f_t\}$ has a \emph{uniform stable radius} (we also say that $\{f_t\}$ is \emph{uniformly stable}) if there exist $\tau > 0$ and $r > 0$ such that for any pair $(\varepsilon,\varepsilon')$ with $0<\varepsilon'\leq \varepsilon\leq r$, there exists $\delta(\varepsilon,\varepsilon')>0$ such that for any non-zero complex number $\eta$ with $0<\vert\eta\vert\leq\delta(\varepsilon,\varepsilon')$, the hypersurface $V(f_t-\eta)$ is non-singular in $\mathring{B}_{r}$ and transversely intersects with the sphere $S_{\varepsilon''}$ for any $ \varepsilon''$ with $\varepsilon'\leq \varepsilon''\leq \varepsilon$ and for any $t$ with $0\leq \vert t\vert \leq\tau$.
Any such a number $r$ is called a \emph{uniform stable radius} for $\{f_t\}$.
\end{definition}

In the special case where the polynomial function $f$ defines a family $\{f_t\}$ of \emph{isolated} hypersurface singularities (i.e., $f_t$ has an isolated singularity at $\mathbf{0}$ for all small $t$), then, by \cite{Milnor2}, we also know that for each $t$ there exists $R_t>0$ such that the hypersurface $V(f_t)$ is non-singular in $\mathring{B}_{R_t}\setminus\{\mathbf{0}\}$ and transversely intersects with the sphere $S_\rho$ for any $\rho$ with $0<\rho\leq R_t$. 

\begin{definition}[\mbox{cf.~\cite[\S 2]{O2}}]\label{def-condA}
Suppose that $f$ defines a family $\{f_t\}$ of \emph{isolated} hypersurface singularities . We say that $\{f_t\}$ satisfies \emph{condition (A)} if there exist $\nu>0$ and $R>0$ such that $V(f_t)$ is non-singular in $\mathring{B}_{R}\setminus\{\mathbf{0}\}$ and transversely intersects with the sphere $S_\rho$ for any $\rho$ with $0<\rho\leq R$ and for any $t$ with $0\leq \vert t\vert \leq\nu$.
\end{definition}

It is easy to see that a family $\{f_t\}$ of isolated hypersurface singularities satisfies condition~\emph{(A)} if and only if it has no \emph{vanishing fold} and no \emph{non-trivial critical arc} in the sense of \cite{OS}. Also, it is worth to observe that if $\{f_t\}$ satisfies condition~\emph{(A)}, then it has a uniform stable radius (cf.~\cite{O2,OS}).

\section{The Oka-O'Shea theorem for isolated singularities}\label{sect-OkaOshea}
Throughout this section we assume that the polynomial function $f$ defines a family $\{f_t\}$ of \emph{isolated} hypersurface singularities.
The following theorem is due to M. Oka \cite{O2} and D. O'Shea \cite{OS}.

\begin{theorem}[Oka-O'Shea]\label{thm-osokaumr}
Suppose that $f$ defines a family $\{f_t\}$ of isolated hypersurface singularities.
Under this assumption, if furthermore $\{f_t\}$ satisfies condition (A) or if it has a uniform stable radius, then it is $\mu$-constant --- that is, the Milnor number $\mu_{f_t}(\mathbf{0})$ of $f_t$ at $\mathbf{0}$ is independent of $t$ for all small $t$.
\end{theorem}

Actually, in \cite{O2}, M. Oka showed that if $\{f_t\}$ satisfies condition 
\emph{(A)} or if it has a uniform stable radius, then the Milnor fibrations at 
$\mathbf{0}$ of $f_0$ and $f_t$ are isomorphic. 

In \cite{LR}, L\^e D\~ung Tr\'ang and C. P. Ramanujam showed that for $n\not=3$ any family of isolated hypersurface singularities with constant Milnor number is topologically $\mathcal{V}$-equisingular. With the same assumption, J. G. Timourian \cite{T} showed that the family is actually topologically trivial.
We recall that a family $\{f_t\}$ is \emph{topologically $\mathcal{V}$-equisingular} (respectively, \emph{topologically trivial}) if there exist open neighbourhoods $D$ and $U$ of the origins in $\mathbb{C}$ and $\mathbb{C}^n$, respectively, together with a continuous map $\varphi\colon (D\times U, D\times \{\mathbf{0}\})
\rightarrow (\mathbb{C}^n,\mathbf{0})$
such that for all sufficiently small $t$, there is an open neighbourhood $U_t\subseteq U$ of $\mathbf{0}\in\mathbb{C}^n$ such that the map 
\begin{equation*}
\varphi_t\colon (U_t,\mathbf{0})\rightarrow (\varphi(\{t\}\times U_t),\mathbf{0}),
\ \mathbf{z}\mapsto\varphi_t(\mathbf{z}):=\varphi(t,\mathbf{z}),
\end{equation*}  
is a homeomorphism satisfying the relation 
\begin{equation*}
\varphi_t(V(f_0)\cap U_t)=V(f_t)\cap \varphi_t(U_t)
\end{equation*}  
(respectively, the relation $f_0=f_t\circ\varphi_t$ on $U_t$).

Note that, in general, ``$\mu$-constant'' does not imply condition \emph{(A)} (cf.~\cite{B,O4}). 

Finally, observe that the Brian\c con-Speder famous family shows that condition~\emph{(A)} does not imply the Whitney conditions along the $t$-axis (cf.~\cite{BS}).

\section{Uniformly stable families of line singularities}

\subsection{Setup and statement of the main result}
From now on we suppose that the polynomial function $f$ defines a family $\{f_t\}$ of hypersurfaces with \emph{line} singularities. As in \cite[\S4]{M7}, by such a family we mean a family $\{f_t\}$ such that for each $t$ small enough, the singular locus $\Sigma f_t$ of $f_t$ near the origin $\mathbf{0}\in\mathbb{C}^n$ is given by the $z_1$-axis, and the restriction of $f_t$ to the hyperplane $V(z_1)$ defined by $z_1=0$ has an isolated singularity at the origin. Then, by \cite[Remark 1.29]{M}, the partition of $V(f_t)$ given by
\begin{equation*}
\mathcal{S}_t:=\bigl\{V(f_t)\setminus\Sigma f_t,\Sigma f_t\setminus\{\mathbf{0}\}, \{\mathbf{0}\}\bigr\}
\end{equation*}
is a \emph{good stratification} for $f_t$ at $\mathbf{0}$, and the hyperplane $V(z_1)$ is a \emph{prepolar slice} for $f_t$ at $\mathbf{0}$ with respect to $\mathcal{S}_t$ for all $t$ small enough.
In particular, combined with \cite[Proposition~1.23]{M}, this implies that the \emph{L\^e numbers}
\begin{equation*}
\lambda^0_{f_t,\mathbf{z}}(\mathbf{0})
\quad\mbox{and}\quad
\lambda^1_{f_t,\mathbf{z}}(\mathbf{0})
\end{equation*}
of $f_t$ at~$\mathbf{0}$ with respect to the coordinates $\mathbf{z}$ do exist. (For the definitions of good stratifications, prepolarity and L\^e numbers, we refer the reader to D.~Massey's book~\cite{M}.) Note that for line singularities, the only possible non-zero L\^e numbers are precisely $\lambda^0_{f_t,\mathbf{z}}(\mathbf{0})$ and $\lambda^1_{f_t,\mathbf{z}}(\mathbf{0})$. All the other L\^e numbers $\lambda^k_{f_t,\mathbf{z}}(\mathbf{0})$ for $2\leq k\leq n-1$ are defined and equal to zero (cf.~\cite{M}). 

Here is our main observation.

\begin{theorem}\label{mt1}
Suppose that $f$ defines a family $\{f_t\}$ of hypersurfaces with line singularities.
Under this assumption, if furthermore $\{f_t\}$ has a uniform stable radius, then it is $\lambda_{\mathbf{z}}$-constant --- that is, the L\^e numbers $\lambda^{0}_{f_t,\mathbf{z}}(\mathbf{0})$ and $\lambda^{1}_{f_t,\mathbf{z}}(\mathbf{0})$ are independent of $t$ for all small $t$.
\end{theorem} 

Theorem \ref{mt1} extends to line singularities Oka-O'Shea's Theorem \ref{thm-osokaumr} concerning isolated singularities. Indeed, for isolated singularities, the only possible non-zero L\^e number is $\lambda^{0}_{f_t,\mathbf{z}}(\mathbf{0})$ and the latter coincides with the Milnor number $\mu_{f_t}(\mathbf{0})$.

Note that if $\{f_t\}$ is a $\lambda_{\mathbf{z}}$-constant family of line singularities in $\mathbb{C}^n$ with $n\geq 5$, then, by a theorem of D. Massey \cite[Theorem (5.2)]{M7}, the diffeomorphism type of the Milnor fibration of $f_t$ at $\mathbf{0}$ is independent of $t$ for all small $t$. Under the same assumption,  in \cite[Theorem 42]{F}, J. Fern\'andez de Bobadilla showed that $\{f_t\}$ is  actually topologically trivial. Combining Fern\'andez de Bobadilla's result with our Theorem \ref{mt1} gives the following corollary.

\begin{corollary}\label{mt2}
Suppose that $f$ defines a family $\{f_t\}$ of hypersurfaces with line singularities in $\mathbb{C}^n$ with $n\geq 5$.
Under this assumption, if furthermore $\{f_t\}$ has a uniform stable radius, then it is topologically trivial.
\end{corollary}

\subsection{Application to families of non-degenerate line singularities with constant Newton boundary}

In \cite[Corollary 1]{O}, M. Oka showed that if $\{f_t\}$ is a family of hypersurface singularities --- not necessary line singularities --- such that for all small $t$ the polynomial function $f_t$ is \emph{non-degenerate} and the \emph{Newton boundary} of $f_t$ at $\mathbf{0}$ with respect to the coordinates $\mathbf{z}$ is independent of $t$, then $\{f_t\}$ has a uniform stable radius. (For the definitions of non-degeneracy and Newton boundary, see \cite{K,O1}.) Combined with Oka's result, Corollary \ref{mt2} provides a new proof of the following theorem of J. Damon \cite{D}. (Actually, the theorem of Damon given in \cite{D} is much more general than the special case stated below.)

\begin{theorem}[Damon]
Suppose that $f$ defines a family $\{f_t\}$ of hypersurfaces with line singularities in $\mathbb{C}^n$ with $n\geq 5$. Under this assumption, if furthermore for any sufficiently small $t$ the polynomial function $f_t$ is non-degenerate and the Newton boundary of $f_t$ at $\mathbf{0}$ with respect to the coordinates $\mathbf{z}$ is independent of $t$, then the family $\{f_t\}$ is topologically trivial.
\end{theorem}

\subsection{Proof of Theorem \ref{mt1}}
Consider the map $\Phi \colon \mathbb{C}\times\mathbb{C}^n\to \mathbb{C}^2$ defined by 
\begin{equation*}
(t,\mathbf{z})\mapsto \Phi(t,\mathbf{z}):=(f(t,\mathbf{z}),t),
\end{equation*}
and pick positive numbers $\tau$ and $r$ which satisfy the condition of Definition \ref{def-WUSR}. Then, in particular, the following property ($\mathcal{P}$) holds true:
\begin{enumerate}
\item [($\mathcal{P}$)]
for any $\varepsilon$ with $0<\varepsilon< r$, there exists $\delta(\varepsilon)>0$ such that for any $t$ with $0\leq\vert t\vert\leq\tau$ and for any $\eta$ with $0<\vert \eta\vert \leq\delta(\varepsilon)$, the hypersurface $V(f_t-\eta)$ is non-singular in $\mathring{B}_r$ and transversely intersects with the sphere $S_{\varepsilon}$.
\end{enumerate}
This property implies that the critical set $\Sigma \Phi$ of $\Phi$ does not intersect with the set
\begin{equation*}
U(\mathring{B}_r):=(\mathring{D}_{\tau}\times \mathring{B}_r) \cap
\Phi^{-1}((\mathring{D}_{\delta(\varepsilon)}\setminus \{\mathbf{0}\})\times\mathring{D}_{\tau}).
\end{equation*}
Indeed, suppose there is a point $(t_0,\mathbf{z}_0)\in\Sigma \Phi\cap 
U(\mathring{B}_r)$. Then $\mathbf{z}_0\in\Sigma (f_{t_0}-f_{t_0}(\mathbf{z}_0))$. 
But this is not possible, since by ($\mathcal{P}$), the hypersurface $V(f_{t_0}-f_{t_0}(\mathbf{z}_0))$ is smooth. (We recall that a complex variety can never be a smooth manifold throughout a neighbourhood of a critical point (cf.~\cite[\S 2]{Milnor2}).)

It also follows from Property ($\mathcal{P}$) that the map
\begin{equation*}
\Phi_{\mid U(S_\varepsilon)} \colon
U(S_\varepsilon)\rightarrow
(\mathring{D}_{\delta(\varepsilon)}\setminus \{\mathbf{0}\})\times\mathring{D}_{\tau}
\end{equation*}
(restriction of $\Phi$ to $U(S_\varepsilon):=(\mathring{D}_{\tau}\times S_\varepsilon) \cap\Phi^{-1}((\mathring{D}_{\delta(\varepsilon)}\setminus \{\mathbf{0}\})\times\mathring{D}_{\tau})$) is a submersion.
Indeed, as $\Sigma \Phi\cap U(\mathring{B}_r)=\emptyset$ and $U(\mathring{B}_r)$ is an open subset of $\mathbb{C}\times\mathbb{C}^n$, the map 
\begin{equation*}
\Phi_{\mid U(\mathring{B}_r)} \colon
U(\mathring{B}_r)\rightarrow (\mathring{D}_{\delta(\varepsilon)}
\setminus \{\mathbf{0}\})\times\mathring{D}_{\tau}
\end{equation*}
is a submersion. Thus, to show that $\Phi_{\mid U(S_\varepsilon)}$ is a  submersion, it suffices to observe that the inclusion $U(S_\varepsilon)\hookrightarrow U(\mathring{B}_r)$ is transverse to the submanifold 
$\Phi_{\mid U(\mathring{B}_r)}^{-1}(f(t,\mathbf{z}),t)$ for any point $(t,\mathbf{z})\in U(S_\varepsilon)$ --- or equivalently that the submanifolds 
\begin{equation*}
\Phi_{\mid U(\mathring{B}_r)}^{-1}(f(t,\mathbf{z}),t)
\quad\mbox{and}\quad
(\{t\}\times S_\varepsilon)\cap U(\mathring{B}_r)
\end{equation*}
are transverse to each other. And this is exactly the content of~($\mathcal{P}$).

Now, as $\Phi_{\mid U(S_\varepsilon)}$ is also a proper map, a result of D. Massey and D. Siersma (cf.~\cite[Proposition 1.10]{MS}) shows that the Milnor number of a generic hyperplane slice of $f_t$ at a point on $\Sigma f_t$ sufficiently close to the origin --- which coincides with the L\^e number $\lambda^{1}_{f_t,\mathbf{z}}(\mathbf{0})$ for line singularities (cf.~\cite{Le,M7}) --- is independent of $t$ for all small $t$. 

Finally, since the family $\{f_t\}$ has a uniform stable radius --- the full strength of this assumption is used here --- it follows from a result of M. Oka (cf.~\cite[Lemma 2]{O}) that the diffeomorphism type of the Milnor fibration of $f_t$ at the origin is independent of $t$ for all small $t$. 
In particular,  the reduced Euler characteristic $\tilde \chi(F_{f_t,\mathbf{0}})$ of the Milnor fibre $F_{f_t,\mathbf{0}}$ of $f_t$ at~$\mathbf{0}$, which is given by
\begin{equation*}
\tilde \chi(F_{f_t,\mathbf{0}})=(-1)^{n-1}\lambda^{0}_{f_t,\mathbf{z}}(\mathbf{0}) + (-1)^{n-2} \lambda^{1}_{f_t,\mathbf{z}}(\mathbf{0})
\end{equation*}
(cf.~\cite[Theorem 3.3]{M}), is independent of $t$ for all small $t$. 
The constancy of $\lambda^{0}_{f_t,\mathbf{z}}(\mathbf{0})$ now follows from that of $\lambda^{1}_{f_t,\mathbf{z}}(\mathbf{0})$.

\section{Uniform stable radius and weighted homogeneous\\ line singularities}

By a result of M. Oka \cite{O2} and D. O'Shea \cite{OS}, we know that if $\{f_t\}$ is a family of \emph{isolated} hypersurface singularities such that each $f_t$ is \emph{weighted homogeneous} with respect to a given system of weights, then $\{f_t\}$ satisfies condition \emph{(A)}, and hence, is uniformly stable.
Our next observation says this still holds true for weighted homogeneous \emph{line} singularities provided that the nearby fibres $V(f_t-\eta)$, $\eta\not=0$, of the functions $f_t$ are ``uniformly'' non-singular with respect to the deformation parameter $t$ --- that is, non-singular in a small ball the radius of which does \emph{not} depends on $t$. (We recall that by \cite{HL} the nearby fibres are ``individually'' non-singular --- that is, non-singular in a small ball the radius of which depends on~$t$.)

\begin{theorem}\label{mt3}
Suppose that $f$ defines a family $\{f_t\}$ of hypersurfaces with line singularities such that each $f_t$ is weighted homogeneous with respect to a given~system of weights $\mathbf{w}=(w_1,\ldots,w_n)$ on the variables $(z_1,\ldots,z_n)$, with $w_i\in\mathbb{N}\setminus\{0\}$. Also, assume that the nearby fibres $V(f_t-\eta)$, $\eta\not=0$, of the functions $f_t$ are uniformly non-singular with respect to the deformation parameter $t$ --- that is, there exist positive numbers $\tau,r,\delta$ such that for any $0<\vert\eta\vert\leq\delta$ and any $0\leq\vert t\vert\leq\tau$, the hypersurface $V(f_t-\eta)$ is non-singular in $\mathring{B}_r$. Under these assumptions, the family $\{f_t\}$ has a uniform stable radius. (In particular, $\{f_t\}$ is $\lambda_{\mathbf{z}}$-constant, and for $n\geq 5$, it is topologically trivial.)
\end{theorem}

\begin{proof}
It is based on similar arguments than those used in \cite{O2} and \cite{OS}.
Suppose that the family $\{f_t\}$ does not have a uniform stable radius. Then, as the nearby fibres of the functions $f_t$ are uniformly non-singular with respect to the deformation parameter $t$, for all $\tau>0$ and all $r>0$ small enough, there exist $0<\varepsilon'\leq\varepsilon\leq r$ such that for all sufficiently small $\delta>0$ there exist $\eta_\delta$, $\varepsilon_\delta$ and $t_\delta$, with $0<\vert\eta_\delta\vert\leq\delta$, $\varepsilon'\leq\varepsilon_\delta\leq\varepsilon$ and $\vert t_\delta\vert\leq\tau$, such that $V(f_{t_\delta}-\eta_\delta)$ is non-singular in $\mathring{B}_r$ and does not transversely intersect with the sphere $S_{\varepsilon_\delta}$. It follows that there is a point $\mathbf{z}_\delta\in V(f_{t_\delta}-\eta_\delta)\cap S_{\varepsilon_\delta}$ which is a critical point of the restriction to $V(f_{t_\delta}-\eta_\delta) \cap B_r$ of the squared distance function:
\begin{equation*}
\mathbf{z}\in V(f_{t_\delta}-\eta_\delta) \cap B_r \mapsto 
\Vert \mathbf{z}\Vert^2=\sum_{1\leq i\leq n} \vert z_i\vert^2.
\end{equation*}
In other words, the point $(t_\delta,\mathbf{z}_\delta)$ lies in the intersection of $D_\tau\times (B_\varepsilon\setminus \mathring{B}_{\varepsilon'})$ with the \emph{real} algebraic set $C$ consisting of the points $(t,\mathbf{z})$ such that
\begin{equation}\label{ld}
\bigg(\frac{\partial f_t}{\partial z_1}(\mathbf{z}),\ldots,\frac{\partial f_t}{\partial z_n}(\mathbf{z})\bigg)=\lambda\bar{\mathbf{z}} 
\end{equation}
for some $\lambda\in\mathbb{C}\setminus\{0\}$,
where $\bar{\mathbf{z}}:=(\bar z_1,\ldots,\bar z_n)$ and $\bar z_i$ denotes the complex conjugate of $z_i$ (see e.g. \cite[Lemma 1]{OS2}). 
Let $C_{\tau,r}:=C\cap(D_\tau\times (B_\varepsilon\setminus \mathring{B}_{\varepsilon'}))$.
Take $\delta:=\delta(m):=1/m$ (where $m\in\mathbb{N}\setminus\{0\}$ is sufficiently large), and consider the corresponding sequence of points $(t_{\delta(m)},\mathbf{z}_{\delta(m)})$ in $C_{\tau,r}$. As $C_{\tau,r}$ is compact, taking a subsequence if necessary, we may assume that $(t_{\delta(m)},\mathbf{z}_{\delta(m)})$ converges to a point $(t_{\tau,r},\mathbf{z}_{\tau,r})\in C_{\tau,r}$, and hence $\eta_{\delta(m)}:=f(t_{\delta(m)},\mathbf{z}_{\delta(m)})$ tends to $f(t_{\tau,r},\mathbf{z}_{\tau,r})$ as $m\to\infty$. Since $0<\vert\eta_{\delta(m)}\vert\leq\delta(m)=1/m\to 0$ as $m\to\infty$, we have $f(t_{\tau,r},\mathbf{z}_{\tau,r})=0$. Thus $(t_{\tau,r},\mathbf{z}_{\tau,r})\in V(f)\cap C_{\tau,r}$.

Now, since $f_{t_{\tau,r}}$ is weighted homogeneous with respect to the weights $\mathbf{w}=(w_1,\ldots,w_n)$, the \emph{Euler identity} implies the following contradiction:
\begin{equation*}
d_{\mathbf{w}}\cdot \underbrace{f_{t_{\tau,r}}(\mathbf{z}_{\tau,r})}_{=0} 
\overset{\mbox{\tiny Euler}}{=}
\sum_{1\leq i \leq n}w_i (z_{\tau,r})_i
\frac{\partial f_{t_{\tau,r}}}{\partial z_i}(\mathbf{z}_{\tau,r}) 
\overset{(\ref{ld})}{=}
\lambda\sum_{1\leq i \leq n}w_i \vert (z_{\tau,r})_i \vert^2 \not=0,
\end{equation*}
where $d_{\mathbf{w}}$ is the weighted degree of $f_{t_{\tau,r}}$ with respect to the weights $\mathbf{w}$ and $(z_{\tau,r})_i$ is the $i$th component of $\mathbf{z}_{\tau,r}$.
\end{proof}

\begin{remark}
Actually, the proof shows that if $f$ defines a family $\{f_t\}$ of hypersurfaces --- not necessarily with line singularities --- such that each $f_t$ is weighted homogeneous with respect to a given~system of weights $\mathbf{w}$, and if furthermore, the nearby fibres $V(f_t-\eta)$, $\eta\not=0$, of the functions $f_t$ are uniformly non-singular with respect to the deformation parameter $t$, then the family $\{f_t\}$ has a uniform stable radius.
\end{remark}

\end{document}